\documentclass{amsart}

\usepackage{amssymb}

\DeclareMathOperator{\dashind}{-Ind}
\DeclareMathOperator{\Prim}{Prim}
\newcommand{\Z}{\mathbb Z}
\newcommand{\N}{\mathbb N}
\newcommand{\C}{\mathbb C}
\newcommand{\h}{\widehat h}
\newcommand{\B}{\mathcal B}

\newtheorem{theorem}{Theorem}[section]
\newtheorem{lemma}[theorem]{Lemma}

\newtheorem{corollary}[theorem]{Corollary}

\theoremstyle{definition}
\newtheorem{definition}[theorem]{Definition}

\theoremstyle{remark}
\newtheorem{remark}[theorem]{Remark}

\numberwithin{equation}{section}

\begin{document}

\title{Topological freeness for  Hilbert bimodules}


\author[B. K. Kwa\'sniewski]{Bartosz  Kosma  Kwa\'sniewski}
\address{Institute of Mathematics,  University  of Bialystok\\
 ul. Akademicka 2,  PL-15-267  Bialystok,
  Poland}
 \email{bartoszk@math.uwb.edu.pl}  
\urladdr{http://math.uwb.edu.pl/~zaf/kwasniewski}
\thanks{This work was in part supported by  Polish National Science Centre  grants numbers  DEC-2011/01/D/ST1/04112, DEC-2011/01/B/ST1/03838.} 


\date{}

\dedicatory{}

\commby{}

\begin{abstract}
It is shown that topological freeness of  Rieffel's induced representation functor  implies that   any $C^*$-algebra generated by a faithful covariant representation of a Hilbert bimodule $X$ over a $C^*$-algebra $A$ is canonically isomorphic to the crossed product $A\rtimes_X \Z$.  An ideal lattice description and  a simplicity criterion for $A\rtimes_X \Z$  are  established. 
\end{abstract}
\maketitle

\section*{Introduction}
The topological freeness is a condition expressed in terms of the dual system allowing to relate the ideal structure  of  the crossed product to that of  the original algebra. In particular,  it implies that every faithful representation of the  $C^*$-dynamical system integrates to a faithful representation of the reduced crossed product.  
The idea behind this notion  probably goes back to works of W. Arveson in late 60's of XX century, and for the first time was explicitly formulated by D. P. O'Donovan,  see \cite[Thm. 1.2.1]{O'Donov}. It is closely related with the properties of the Connes spectrum \cite{Olesen-pedersen} and with proper outerness \cite{elliot},  \cite{AS93}. The role of topological freeness for $C^*$-dynamical systems with arbitrary discrete group actions  on commutative $C^*$-algebras was clarified in \cite[Thm. 4.4]{kaw-tom} and  for  arbitrary $C^*$-algebras  in \cite[Thm. 1]{AS93}. Independently, and even earlier, in connection with investigation of spectral properties of functional operators, similar results were obtained by A. B. Antonevich, A. V. Lebedev and others,  see   \cite[Cor. 12.17]{AL94}  and \cite[pp. 225, 226]{AL94} for the corresponding survey. These results were improved to cover the case of partial actions in  \cite[Thm. 2.6]{exel3} and \cite[Thm. 3.7]{top-free}.

In the present paper we  prove a statement that generalizes all the aforementioned theorems in the case $G=\Z$ and which is formulated in terms of the crossed product $A\rtimes_X \Z$, introduced in  \cite{AEE98}, of a Hilbert bimodule $X$. Thus potentially, by passing to the core $C^*$-algebra, see  \cite[Thm. 3.1]{AEE98}, it may be applied to  all the $C^*$-algebras  equipped with a semi-saturated circle action and thereby to all relative Cuntz-Pimsner algebras \cite{ms}.
 As a corollary of our main result we  provide  an ideal lattice description (in the case the dual system is free) and a simplicity criterion for the  algebras considered.

\subsubsection*{Conventions} In essence we follow the notation and conventions adopted in  \cite{AEE98}. For  maps $\gamma\colon A\times B\to C$ 
such as inner products, multiplications or representations  we denote by $\gamma(A,B)$  the closed linear span of the set
$\{\gamma(a,b)\in C : a\in A,b\in B\}$. An ideal in a $C^*$-algebra is a closed two-sided one, and $[\pi]$ stands for the unitary equivalence class of a representation $\pi$.

\section{Hilbert bimodules, their crossed products and dual partial dynamical systems}

 Let $A$ be a $C^*$-algebra and  $X$ be an   $A$-$A$-Hilbert bimodule 
 as introduced in \cite[1.8]{BMS}. More specifically, $X$ is both left and right Hilbert module over $A$ with  left and right  $A$-valued inner products $\langle x,y\rangle_L$ and  $\langle x,y\rangle_R$
satisfying the so-called imprimitivity condition:
$$
 x \cdot \langle y ,z \rangle_R = \langle x , y  \rangle_L \cdot z, \qquad \textrm{for all}\,\,\, x,y,z\in X.
 $$ 
  A simple but crucial observation is that $X$ may treated as  an imprimitivity $I_L -I_R$-bimodule where
$$
I_L=\langle X , X  \rangle_L, \qquad I_R= \langle X , X  \rangle_R
$$ 
are  ideals in $A$.  Hence  the induced representation functor $X\dashind=X\dashind_{I_R}^{I_L}$ factors through
to a homeomorphism $\h:\widehat{I_R} \to \widehat{I_L}$ between the spectra  of $I_L$ and $I_R$:
$$
\h([\pi]):= [X\dashind_{I_L}^{I_R}\pi],
$$
 cf. e.g.  \cite[Thm. 3.29, Cor. 3.32, 3.33]{morita}. 

Thus  identifying the spectra $\widehat{I_L}$ and $\widehat{I_R}$ with open subsets of the spectrum   of $A$, we may and we will treat $\h$ as a mapping between open subsets of $\widehat{A}$.
\begin{definition}
We call the partial homeomorphism $\h$ of $\widehat{A}$ described above a \emph{partial homeomorphism  dual to the Hilbert bimodule} $X$.
\end{definition}
 A \emph{covariant representation }of $(A,X)$ \cite[Defn. 2.1]{AEE98} is a pair $(\pi_A,\pi_{X})$ of representations into algebra of all bounded linear operators $\B(H)$ on a Hilbert space $H$ such that all module operations become the ones inherited form $\B(H)$, i.e. 

 $$\pi_X(ax)=\pi_A(a)\pi_X(x),\qquad \pi_X(xa)=\pi_X(x)\pi_A(a),
$$
$$
 \pi_A(\langle x ,y \rangle_R)=\pi_X(x)^*\pi_X(y),\qquad  \pi_A( \langle x , y  \rangle_L)=\pi_X(x)\pi_X(y)^*,
$$

 for all $a\in A$, $x,y \in X$.  We say that $(\pi_A,\pi_{X})$ is faithful if $\pi_A$ is faithful (then $\pi_X$ is automatically isometric).   A \emph{crossed product} of $A$ by $X$, see \cite[Defn. 2.4]{AEE98},   is the  $C^*$-algebra $A\rtimes_X \Z$ universal with respect to  covariant representations of $(A,X)$.
We denote by $\pi_A \rtimes \pi_X$ the representation of $A\rtimes_X \Z$ corresponding to  $(\pi_A,\pi_X)$ and call it an integrated form of $(\pi_A,\pi_X)$. 

The interior tensor power  $X^{\otimes n}$, $n\geq 1$, of $X$ is naturally an $A$-$A$ Hilbert bimodule which,  as follows from the lemma below,  embeds  into $A\rtimes_X \Z$.
\begin{lemma}\label{representation structure lemma}
Suppose $(\pi_A,\pi_X)$ is a covariant representation of $X$. Then for every $n\in \N$, the mapping  $$
X^{\otimes n} \ni x_1\otimes x_2 \otimes ... \otimes x_n \stackrel{\pi_{X^{\otimes n}}}{\longrightarrow} \pi_{X}(x_1)\pi_{X}(x_2)... \pi_{X}(x_n).
$$
yields a well defined linear homomorphism such that the pair $(\pi_A,\pi_{X^{\otimes n}})$ is a covariant representation of the tensor product Hilbert bimodule $X^{\otimes n}$. In particular, the linear span of the spaces
$$
\pi_{X^{\otimes n}}(X^{\otimes n}), \qquad \pi_A(A), \qquad \pi_{X^{\otimes n}}(X^{\otimes n})^*, \qquad n\in \N,
$$
forms a dense $^*$-subalgebra of $(\pi_A \rtimes \pi_X)(A\rtimes_X \Z)=C^*(\pi_A(A)\cup \pi_X(X))$.
\end{lemma}
\begin{proof}
Apply for instance \cite[Lem. 2.7]{AEE98}, see also \cite[Lem. 2.5]{AEE98}.
\end{proof}

The next lemma shows  how to iterate representation $[\pi] \in \widehat{A}$ under  $\h$  using representations of   $A\rtimes_X \Z$. Roughly, in order to determine $\h^n([\pi])$  it suffices to  extend $\pi$ (say, acting   in a Hilbert space $K$) to any representation $\nu:A\rtimes_X \Z\to \B(H)$, $K\subset H$, and then determine representation $\nu|_A$ of $A$ acting in the subspace  $\nu(X^{\otimes n})K$. 

\begin{lemma}\label{representation structure lemma2}
Suppose $(\pi_A,\pi_X)$ is covariant representation of $X$ in a Hilbert space $H$ and let $\pi$ be an irreducible summand of $\pi_A$ acting on a Hilbert subspace $K$. The representation $\pi_n:A\to \B(\pi_{X^{\otimes n}}(X^{\otimes n})K)$ where 
$$
\pi_n(a):= \pi_A(a)|_{\pi_{X^{\otimes n}}(X^{\otimes n})K}, \qquad a\in A,
 $$ 
 is non-zero if and only if $[\pi]$ belongs to the domain of $\h^n$, and then 
 $$
 [\pi_n]= \h^n([\pi]).
 $$
 \end{lemma}
\begin{proof}
We recall  that  $X\dashind(\pi)(a)  (x\otimes_\pi h) = (a x)\otimes_\pi h$ where $X\otimes_\pi H$ is the tensor product Hilbert space with the inner product satisfying
$\langle x_1\otimes_\pi h_1, x_2\otimes_\pi h_2 \rangle = \langle h_1,\pi(\langle x_1, x_2 \rangle_{R})h_2\rangle$. In particular, one sees that $[X^{\otimes n}\dashind(\pi)]=\h^n([\pi])$. Since  $(\pi_A,\pi_{X^{\otimes n}})$ is a covariant representation of $X^{\otimes n}$ we have
$$
\|\pi_{X^{\otimes n}}(x)h\|^2=\langle \pi_{X^{\otimes n}}(x)h,\pi_{X^{\otimes n}}(x)h \rangle =\langle h,\pi_A(\langle x, x\rangle_R) h \rangle= \|x\otimes_{\pi} h\|^2.
$$
Accordingly, the mapping      $\pi_{X^{\otimes n}}(x)h \mapsto x \otimes_\pi h$, $x\in X^{\otimes n}$, $h\in K$, extends by linearity and continuity to a unitary  operator $V:\pi_{X^{\otimes n}}(X^{\otimes n})K \to X^{\otimes n} \otimes_\pi K$, which  intertwines $\pi_n$ and $X^{\otimes n}\dashind(\pi)$ because 
$$
V\pi_n(a)\pi_{X^{\otimes n}}(x)h =V \pi_{X^{\otimes n}}(ax) h= (ax \otimes_\pi h)= {\h}^n(\pi)(a)  V \pi_{X^{\otimes n}}(x)h.
$$
\end{proof}
\section{The main result  and its corollaries}
 The goal of the paper could be stated as follows.

\begin{definition}\label{top-free} We say that  a partial homeomorphism $\varphi$ of a topological space, i.e. a homeomorphism between open subsets,   is   {\em
topologically free} if  for any $n>0$ the set $F_n=\{x : \varphi^n(x)=x\}$ (contained in the domain of $\varphi^n$) has empty interior.
\end{definition}

\begin{theorem}\label{main result}
If the partial homeomorphism $\h$ is  topologically
free, then    every faithful  covariant representation  $(\pi_A,\pi_X)$   of  $X$ integrates to
faithful representation $(\pi_A \times \pi_X)$ of $A\rtimes_X \Z$.
\end{theorem}
\begin{remark}\label{remark about domains}
The map $\h$ is a lift of the partial homeomorphism  $h:\Prim (I_R)\to \Prim(I_L)$ of  $\Prim(A)$ where  $h(\ker \pi):= \ker X\dashind(\pi)$, $[\pi] \in\widehat{A}$. Actually $h$ is the restriction of the so-called Rieffel isomorphism between the ideal lattices of $I_R$ and  $I_L$ where 
\begin{equation}\label{Rieffel homoemorphism}
  h(J)= \langle X J , X  \rangle_L , \qquad  h^{-1}(K)=\langle X, KX  \rangle_R,
\end{equation}
cf. \cite[3.3]{morita}.
Plainly, topological freeness of the Rieffel homeomorphism  $h$, treated as a partial homeomorphism of $\Prim(A)$ implies the topological freeness of $\h$. However, the converse implication is not true.
\end{remark}

An   equivalent form of  Theorem \ref{main result} states that   if the partial homeomorphism $\h$  is  topologically
free, then every non-trivial ideal in  $A\rtimes_X \Z$ leaves an "imprint" in $A$ -- has a non-trivial intersection with $A$. By specifying these imprints one may determine the  ideal structure of $A\rtimes_X \Z$. To this end we adopt the following definition, cf. \cite[Defn. 2.7, 2.8]{exel3}.
\begin{definition}\label{minim} We say that  a set $V$ is \emph{invariant} under a partial homeomorphism $\varphi$ with a domain $\Delta$ if 
$$
\varphi(V\cap \Delta)= V\cap \varphi(\Delta). 
$$
If there are no non-trivial closed invariant sets, then   $\varphi$   is  called  {\em
minimal}, and  $\varphi$  is said to be \emph{free}, if it is topologically free on every closed invariant set (in the  Hausdorff space case   this amounts to requiring that $\varphi$ has no periodic points).
\end{definition}
Similarly to topological freeness, cf. Remark \ref{remark about domains}, the freeness of $h$ is a stronger condition than freeness of $\h$. However,  the minimality of $\h$ and $h$ are equivalent, and moreover using \eqref{Rieffel homoemorphism} one sees that, if $I$ is an ideal in $A$, then the open set $\widehat{I}$  in $\widehat{A}$ is $\h$-invariant if and only if 
 \begin{equation}\label{invariant ideals condition}
IX=XI.
\end{equation}
Ideals satisfying \eqref{invariant ideals condition} are called $X$-invariant in   \cite{katsura2}, and  $X$-invariant and saturated in \cite{Kwa}.
It is known, see \cite[10.6]{katsura2} or \cite[Thm. 7.11]{Kwa}, that 
the map 
\begin{equation}\label{Rieffel circle action}
A\rtimes_X \Z \triangleright J \longmapsto  J \cap A \triangleleft A
\end{equation}
defines a homomorphism from the  lattice of ideals in $A\rtimes_X \Z$  onto the lattice of ideals  satisfying \eqref{invariant ideals condition}.  When restricted to \emph{gauge invariant ideals}, i.e. ideals preserved under the \emph{gauge circle action} $\{\lambda\in  \C:|\lambda|=1\}\ni  \lambda \to \gamma_\lambda \in Aut(A\rtimes_X \Z)$   given  by
\begin{equation}\label{circle action}
 \gamma_\lambda(a)= a,\quad a\in A,\qquad  \gamma_\lambda(x)= \lambda x, \quad x\in X,
\end{equation}
 homomorphism  \eqref{Rieffel circle action} is actually an isomorphism. Thus if one is able to show that all ideals in $A\rtimes_X \Z$ are gauge invariant, one obtains a complete description of the ideal structure of $A\rtimes_X \Z$.

\begin{theorem}[ideal lattice description]\label{main result3}
If  the partial homeomorphism $\h$ is free, then  all ideals in $A\rtimes_X \Z$ are gauge invariant and the map
\begin{equation}\label{lattice isomorphism}
J \mapsto  \widehat{J \cap A}
\end{equation}
is a lattice isomorphism between ideals in $A\rtimes_X \Z$ and  open invariant sets in $\widehat{A}$. 

\end{theorem}
\begin{proof}
It suffices to show  that  the map \eqref{lattice isomorphism} is injective. To this end suppose  that  $J$ is an ideal in $A\rtimes_X \Z$, let $J_0=J\cap A$ and denote by $\langle J_0 \rangle$  the ideal generated by $J_0$ in $A\rtimes_X \Z$.  Clearly, $\langle J_0 \rangle\subset J$  and  to  prove that  $\langle J_0 \rangle= J$ we consider the homomorphism
$$
\Psi: A\rtimes_X \Z\to  A/J_0\rtimes_{X/XJ_0} \Z
$$
arising from the composition of the quotient maps and the  universal covariant representation of $(A/J_0,X/XJ_0)$. Then, cf. for instance \cite[Thm. 6.20]{Kwa}, $\ker \Psi=\langle J_0 \rangle$ and we claim that $\Psi(J)\cap   A/J_0=\{0\}$. Indeed, if $b\in \Psi(J)\cap   A/J_0$,   then $b = \Psi(a)$ for some $a \in J$ and
$b =\Psi(a_1)$ for some $a_1\in A$. Thus $a-a_1 \in \ker\Psi=\langle J_0 \rangle\subset J$ and it follows
that $a_1$ itself is in J. But then $a_1\in  J \cap A = J_0$, so $b =\Psi(a_1) = 0$, which proves our claim.
The system dual to $(A/J_0, X/XJ_0)$ naturally identifies with $(\widehat{A}\setminus \widehat{J_0},\h)$ and thus  it is topologically free by freeness of $(\widehat{A},\h)$. Hence  Theorem \ref{main result} implies that $\Psi(J)$ is trivial in $A/J_0\rtimes_{X/XJ_0} \Z$. Hence $J=\langle J_0 \rangle= \ker\Psi$.
\end{proof}
\begin{corollary}[simplicity criterion]\label{main result2}
If the partial homeomorphism $\h$ is  topologically
free and minimal, then $A\rtimes_X \Z$ is simple.
\end{corollary}
\section{Proof of Theorem \ref{main result}} 

 We denote by $X_n$, $n\in \Z$,  the fibers of the $\Z$-bundle constructed in \cite[p. 3046-3047]{AEE98}. In particular, $X_{0}=A$  and  for $n>0$, $X_n=X^{\otimes n}$   and  $X_{-n}=\widetilde{X}^{\otimes |n|}$, where  $\widetilde{X}$ is the Hilbert  $I_R -I_L$-bimodule   dual to the $I_L -I_R$-bimodule $X$.  The $C^*$-algebraic bundle operations  equip   each pair $(A,X_n)$, $n\in \Z$,   with a Hilbert bimodule structure. 
For all  $n\in \Z$ 
we  put
$$
D_n:={\langle X_n, X_n \rangle}_R=  X_n^*\cdot  X_n
$$
where ``$\cdot$'' is the $C^*$-algebraic bundle multiplication (then, in particular, $D_0=A$ and $D_{-n} =   {\langle X_n, X_n \rangle}_L$). We notice that  $X_n$ is a $D_{-n}- D_{n}$-imprimitivity bimodule and the partial homeomorphism  of $\widehat{A}$ given by the induced representation functor  $X_n\dashind^{D_{-n}}_{D_n}$ coincides with the $n$-th power $\h^n$ of $\h$, and in particular, $\widehat{D}_{n}$ is a natural domain of $\h^n$. 

A covariant representation $(\pi_A,\pi_{X})$ of $(A,X)$   yields   covariant representations $(\pi_A,\pi_{X_n})$ of $(A,X_{n})$ for all $n\in \Z$, cf.  Lemma \ref{representation structure lemma},  \cite[Lem. 2.7]{AEE98} or \cite[Thm. 3.12]{Kwa}.  The copy of $A$ embedded into $A\rtimes \Z$ is a fixed point algebra for the gauge circle action \eqref{circle action}. Thus, in the language of circle actions, it seems to be a part of a $C^*$-algebra folklore and follows for instance from \cite[Lem. 2.11]{DR} or \cite[Lem. 2.2]{BKR93}  that the representation $(\pi_A \times \pi_X)$ of $A\rtimes_X \Z$ is faithful if and only if $\pi_A$ is faithful and the formula 
$$
\mathcal{E}\left(\sum_{k=-n}^{n} \pi_{X_k}(a_{k}) \right)=\pi_A(a_0), 
$$
 where $a_{ k}\in X_{k},\,\, k=0, \pm 1,...,\pm n$,
defines  a mapping (conditional expectation) $\mathcal{E}$ from  the $C^*$-algebra $ C^*(\pi_A(A),\pi_X(X))$ generated by $\pi_A(A)$ and $\pi_X(X)$ onto the $C^*$-algebra $\pi_A(A)$.
Therefore  Theorem \ref{main result}   follows immediately from Lemma \ref{2.6} below, and among the technical instruments of the proof
of this latter statement  we use the following simple fact, see e.g. \cite[Lem. 12.15]{AL94}.
\begin{lemma}
\label{a}
Let $B$ be a $C^*$-subalgebra of an algebra
$\B(H)$ and  $P_1$, $P_2  \in  B^\prime$ be  two orthogonal projections such that the restrictions of $B$ to  $H_{1} = P_1 H$ and $H_{2} = P_2 H$ are both irreducible representations. Then $P_1\neq P_2$ if and only if $P_{1} \perp  P_{2}$.
\end{lemma}
\begin{lemma}
\label{2.6} Let the Rieffel homeomorphism $\h$ be 
 topologically
free.  Assume that $A$ and $X$ are embedded in $\B(H)$ so that the module actions and inner products become inherited from $\B(H)$ (then the whole $\Z$-bundle $\{X_n\}_{n\in \Z}$ embeds into  $\B(H)$)
and  let $b$ be an  operator   of the form
\begin{equation}
\label{e2.1a}
b= \sum_{k=-n}^{n} a_{k} \qquad \textrm{ where  }\,\,\, a_{ k}\in X_k,\,\, k=0, \pm 1,...,\pm  n.
\end{equation}
Then for every $\varepsilon >0$ there exists an irreducible representation
$\pi:A\to \B(H_{\pi })$ such that for any 
 representation $\nu: C^* (A,X)\to \B(H_\nu)$
 that  extends $\pi$  ($H_{\pi}\subset  H_\nu$) we have
\begin{itemize}
\item[(i)]\ \ $\Vert \pi (a_0)\Vert \ge \Vert a_0 \Vert  - \varepsilon$,
\item[(ii)]\ \ $  P_{\pi } \,   \pi (a_0)\, P_{\pi } = P_{\pi } \, \nu (b)\,  P_{\pi }  $,
\end{itemize}
where $P_{\pi }\in \B(H_\nu)$ is the orthogonal projection onto $H_{\pi }$.
\end{lemma}
\begin{proof} Let $\varepsilon >0$.  Since  for every $a\in A$ the function $[\pi] \to \|\pi(a)\|$ is lower semicontinuous on $\widehat{A}$ and attains its upper bound equal to $\Vert a \Vert$, cf. for instance \cite[App. A]{morita}, there exists an open set
$U \subset \widehat{A}$ such that
$$
\|\pi(a_0)\| >  \Vert a_0 \Vert - \varepsilon \ \ {\rm for\ \  every}\ \ [\pi]\in U.
$$
 By topological freeness  of $\h$ the set $F_{n!}=\{ [\pi] \in \widehat{D}_{n!}: \h^{n!}([\pi])=[\pi]\}$ has empty interior and thus we may find
$[\pi] \in U $ such that all the points $\h^{k}([\pi] )$,  $k=0,1,...,n$ are   distinct (if they are
defined, i.e. if $\pi(D_{k})\neq 0$).
Let  $\nu$ be any extension of  $\pi$ up to a representation of   $C^* (A,X)$ and denote by $H_{\pi}$  and $H_\nu$  the corresponding representation spaces 
for $\pi$ and $\nu$:
$
H_\pi\subset H_\nu.
$
Item (i) follows from the choice of $\pi$. 
To prove (ii) we need to show  that for 
  the orthogonal projection
$P_\pi:H_\nu\to H_\pi$ and any  element  $a_{k}\in X_{k}$,
$k\neq 0$,  of the sum \eqref{e2.1a} we have
$$
P_{\pi}\,  \nu(a_{k}) \, P_{\pi} =0.
$$
We fix $k\neq 0$ and consider  two  different possible positions of $\pi$.
\par
If $\pi \notin \widehat{D}_{k}\cap \widehat{D}_{-k}$, then  either  $\pi(D_k)=0$ or $\pi(D_{-k})=0$. By Hewitt-Cohen Theorem  (see, for example, \cite[Prop. 2.31]{morita}) operator $a_k$ may be presented in a form $a_k=d_{-}ad_{+}$ where $d_{\pm}\in D_{\pm k}$, $a\in X_k$, and thus 
$$
P_{\pi}\,  \nu(a_{k}) \, P_{\pi} =P_{\pi}\,  \nu(d_{-}a d_+) \, P_{\pi}  =  P_{\pi}\, \pi(d_-) P_{\pi}\nu(a) P_{\pi}\pi(d_+) \, P_{\pi}=0.
$$
\par
Suppose then that $\pi \in \widehat{D}_{k} \cap \widehat{D}_{-k}$. Accordingly, $\pi$ may be treated as an irreducible representation for both $D_k$ and $D_{-k}$. 
We will use Lemma \ref{a} where the role of $P_1$ is played by $P_{\pi}$ and $P_2$ is the orthogonal projection onto the space
$$
H_2:=\nu(X_k)H_\pi.
$$
Clearly,  $P_\pi \in \nu(A)'$ and to see that  $P_2\in \nu(A)'$ it suffices to note that for $a\in A$, $x\in X_k$ and $h\in H_\pi$ we have 
$
\nu(a) \nu(x)h= \nu(ax)h \in H_2
$, that is $\nu(a)P_2=P_2 \nu(a) P_2$, since then  using the same relation for $\nu(a^*)$ one gets $\nu(a)P_\pi=P_\pi\nu(a)$. Moreover, by Lemma \ref{representation structure lemma2} for the representation $\pi_2:A\to L(H_2)$ given by $\pi_2(a)=\nu(a)|_{H_2}$, we have $\pi_2\cong X_{k}\dashind(\pi)$, or equivalently
$$
[\pi_2]={\h}^k([\pi]).
$$
Consequently,  $\pi$ and $\pi_2$ may be treated as irreducible representations of $D_{-k}$, and by the choice of $\pi$ these representations are different (actually even not equivalent). Hence, by Lemma \ref{a}
$$
P_\pi\cdot P_2 =0
$$
from which we have
$$
P_{\pi}\,  \nu(a_{k}) \, P_{\pi}=P_{\pi}\cdot P_2\, \nu(a_{k}) \, P_{\pi}=0.
$$
\end{proof}

\bibliographystyle{plain}

\begin{thebibliography}{10}

\bibitem{AEE98}
B.~Abadie, S.~Eilers, and R.~Exel.
\newblock Morita equivalence for crossed products by {H}ilbert
  ${C}^*$-bimodules.
\newblock {\em Trans. Amer. Math. Soc.}, 350(8):3043--3054, 1998.

\bibitem{AL94}
A.~B. Antonevich and A.~V. Lebedev.
\newblock {\em Functional differential equations: I. ${C}^*$-theory}.
\newblock Longman Scientific \& Technical, Harlow, Essex, England, 1994.

\bibitem{AS93}
R.~J. Archbold and Spielberg~J. S.
\newblock Topologically free actions and ideals in discrete ${C}^*$ -dynamical
  systems.
\newblock {\em Proc. Edinburgh Math. Soc.}, 37(2):119--124, 1993.

\bibitem{BKR93}
S.~Boyd, N.~Keswani, and I.~Raeburn.
\newblock Faithful representations of crossed products by endomorphisms.
\newblock {\em Proc. Amer. Math. Soc.}, 118:427--436, 1993.

\bibitem{BMS}
L.~G. Brown, J.~Mingo, and N.~Shen.
\newblock Quasi-multipliers and embeddings of hilbert ${C}^*$-modules.
\newblock {\em Canad. J. Math.}, 71:1150--1174, 1994.

\bibitem{DR}
S.~Doplicher and J.~E. Roberts.
\newblock Duals of compact lie groups realized in the {C}untz algebras and
  their actions on ${C}^*$-algebras.
\newblock {\em J. Funct. Anal.}, 74:96--120, 1987.

\bibitem{elliot}
G.~A. Elliot.
\newblock Some simple ${C}^*$-algebras constructed as crossed products with
  discrete outer automorphism groups.
\newblock {\em Publ. Res. Inst. Math. Sci.}, 13:299--311, 1980.

\bibitem{exel3}
R.~Exel, M.~Laca, and J.~Quigg.
\newblock Partial dynamical systems and ${C}^*$-algebras generated by partial
  isometries.
\newblock {\em J. Operator Theory}, 47:169--186, 2002.

\bibitem{katsura2}
T.~Katsura.
\newblock Ideal structure of ${C}^*$-algebras associated with
  ${C}^*$-correspondences.
\newblock {\em Pacific J. Math.}, 230:107--146, 2007.

\bibitem{kaw-tom}
S.~Kawamura and J.~Tomiyama.
\newblock Properties of topological dynamical systems and corresponding
  ${C}^*$-algebras.
\newblock {\em Tokyo J. Math.}, 13:251--257, 1990.

\bibitem{Kwa}
B.~K. Kwa\'sniewski.
\newblock ${C}^*$-algebras generalizing both relative {C}untz-{P}imsner and
  {D}oplicher-{R}oberts algebras.
\newblock accepted in Trans. Amer. Math. Soc., arXiv:0906.4382.

\bibitem{top-free}
A.~V. Lebedev.
\newblock Topologically free partial actions and faithful representations of
  partial crossed products.
\newblock {\em Funct. Anal. Appl.}, 39:207--214, 2005.

\bibitem{ms}
P.~S. Muhly and B.~Solel.
\newblock Tensor algebras over ${C}^*$-correspondences (representations,
  dilations, and ${C}^*$-envelopes).
\newblock {\em J. Funct. Anal.}, 158:389--457, 1998.

\bibitem{O'Donov}
D.~P. O'Donovan.
\newblock Weighted shifts and covariance algebras.
\newblock {\em Trans. Amer. Math. Soc.}, 208:1--25, 1975.

\bibitem{Olesen-pedersen}
D.~Olesen and G.~K. Pedersen.
\newblock Applications of the {C}onnes spectrum to ${C}^*$-dynamical systems.
\newblock {\em J. Funct. Anal.}, 30:179--197, 1978.

\bibitem{morita}
I.~Raeburn and D.~P. Williams.
\newblock {\em Morita equivalence and continuous-trace ${C}^*$-algebras}.
\newblock Amer. Math. Soc., Providence, 1998.

\end{thebibliography}


\end{document}